\documentclass[12pt]{amsart}
\usepackage{amsthm}
\usepackage{amssymb}
\usepackage{amsmath}
\usepackage{tikz}
\usepackage{url}
\usepackage{cite}
\usepackage[colorinlistoftodos]{todonotes}
\nocite{*}
\usepackage{a4wide}

\theoremstyle{definition} 
\newtheorem{lemma}[subsection]{Lemma}

\newtheorem{proposition}[subsection]{Proposition}

\newtheorem{definition}[subsection]{Definition}

\newtheorem{theorem}[subsection]{Theorem}
\newtheorem{fact}[subsection]{Fact}

\newtheorem*{rep@theorem}{\rep@title}
\newcommand{\newreptheorem}[2]{%
\newenvironment{rep#1}[1]{%
 \def\rep@title{#2 \ref{##1}}%
 \begin{rep@theorem}}%
 {\end{rep@theorem}}}
\makeatother

\newreptheorem{theorem}{Theorem}
\newreptheorem{lemma}{Lemma}

\newcommand{\comment}[1]{}

\title{Local definable $(p,q)$-theorem and a refinement}
\author{Gareth John Boxall, Charlotte Kestner, Tsinjo Rakotonarivo}
\address{Department of Mathematical Sciences (Mathematics Division), University of Stellenbosch, South Africa}
\email{gboxall@sun.ac.za}
\address{Department of Mathematics, Imperial College, London}
\email{c.kestner@imperial.ac.uk}
\address{Department of Mathematics and Applied Mathematics, University of Pretoria, South Africa and DSI-NRF Centre of Excellence in Mathematical and Statistical Sciences (CoE-MaSS), South Africa}
\email{tsinjo@aims.ac.za}
\thanks{The support of the DSI-NRF Centre of Excellence in Mathematical and Statistical Sciences (CoE-MaSS) towards this research is hereby acknowledged. Opinions expressed and conclusions arrived at are those of the authors and are not necessarily to be attributed to the CoE}

\begin{document}

\begin{abstract}
We prove the definable $(p,q)$-conjecture for NIP formulas in arbitrary first-order theories. We also prove a refinement and apply it to dense pairs. 
\end{abstract}

\maketitle

\section{Introduction}

We prove the definable $(p,q)$-conjecture for NIP formulas. This is a first-order definable version of a theorem of Matou\v{s}ek concerning set systems with bounded Vapnik-Chervonenkis dimension. 

A set system is a pair $(X,\mathcal{S})$, where $X$ is a non-empty set and $\mathcal{S}$ is a non-empty collection of subsets of $X$. The dual Vapnik-Chervonenkis dimension $\text{VC}^*(X,\mathcal{S})$ is the maximum $n$ for which there exist $Y_1,...,Y_n\in \mathcal{S}$ such that each of the $2^n$ fields of the Venn diagramme for $Y_1,...,Y_n$ is non-empty, unless there is no such maximum, in which case $\text{VC}^*(X,\mathcal{S})=\infty$. 

\begin{theorem}(Matou\v{s}ek, \cite{Matousek2004})\label{pq}
Let $q\leq p$ be positive integers. There exists a positive integer $n$ such that, for every non-empty set $X$ and finite non-empty set $\mathcal{S}$ of non-empty subsets of $X$, if \begin{enumerate}
    \item $\text{VC}^*(X,\mathcal{S})<q$ and
    \item for every $p$ elements of $\mathcal{S}$ some $q$ of them have non-empty intersection
\end{enumerate} then there exist $a_1,...,a_n\in X$ such that, for every $Y\in \mathcal{S}$, there is some $i\in \{1,...,n\}$ such that $a_i\in Y$.
\end{theorem}

Assumption (2) is referred to as the $(p,q)$-property. Matou\v{s}ek's result builds on work of Alon and Kleitman \cite{Alon1992} in which the same conclusion is obtained for $X=\mathbb{R}^d$ and $\mathcal{S}$ a set of compact convex subsets of $\mathbb{R}^d$ with the $(p,q)$-property, assuming $p\geq q>d$. The result of Alon and Kleitman solved a long-standing problem of Hadwiger and Debrunner. 

Our interest in Theorem \ref{pq} comes from model theory, where it has had striking applications. It was used by Chernikov and Simon in \cite{Chernikov2015} to prove the existence of uniform honest definitions in NIP theories and, consequently, Laskowski's uniform definability of types over finite sets (UDTFS) conjecture for NIP theories. Then in \cite{Bays2022}, Bays, Kaplan and Simon used it to prove the existence of uniform honest definitions for NIP formulas in arbitrary first-order theories. (A partitioned first-order formula $\varphi(x,y)$ has NIP relative to a first-order theory $T$ if, in any model $M\models T$, the system of sets defined by the formulas $\varphi(x,b)$, as $b$ varies in a sort of $M$, has finite dual VC-dimension. A theory has NIP if all formulas do.)

Theorem \ref{pq} is closely connected to another conjecture in model theory, concerning the notion of dividing. Simon notes in \cite{Simon2014} that this model-theoretic conjecture amounts to a first-order definable version of Theorem \ref{pq}. It became known as the definable $(p,q)$-conjecture and has been proved at various levels of generality, by various authors. Building on this work, we prove the definable $(p,q)$-conjecture for NIP formulas in arbitrary first-order theories (the local definable $(p,q)$-conjecture). The statement is as follows.

\begin{theorem}\label{mt version}
Let $L$ be a language and let $M\prec N$ be $L$-structures, with $N$ sufficiently saturated. Let $\varphi(x,y)$ be a first-order $L$-formula that has NIP with respect to $\text{Th}(M)$. Let $b\in N$ and suppose $\varphi(x,b)$ does not divide over $M$. Then there is a formula $\psi\in \text{tp}(b/M)$ such that $\{\varphi(x,b'):b'\in\psi(M)\}$ is realised in $N$.
\end{theorem}

To understand the connection between Theorem \ref{mt version} and Theorem \ref{pq}, we should state a version of the latter for infinite set systems. Simply allowing $\mathcal{S}$ to be infinite in Theorem \ref{pq} does not work. For example, let $X=\mathbb{Q}$ and let $\mathcal{S}$ be the set of all intervals $(a,b)\subseteq\mathbb{Q}$ such that $a<\sqrt{2}<b$. Then $\text{VC}^*(X,\mathcal{S})=2$ and every three elements of $\mathcal{S}$ have non-empty intersection, so $(X,\mathcal{S})$ has the $(3,3)$-property. However, the conclusion of Theorem \ref{pq} is false. Using the compactness theorem of first-order logic, it is easy to see that the following is equivalent to Theorem \ref{pq}. 

\begin{theorem}\label{infinite pq}
Let $X$ be a non-empty set and let $\mathcal{S}$ be a non-empty (not necessarily finite) set of non-empty subsets of $X$. Let $q\leq p$ be positive integers. Suppose $\text{VC}^*(X,\mathcal{S})<q$ and $(X,\mathcal{S})$ has the $(p,q)$-property. There exist an integer $n$, depending only on $(p,q)$, and subsets $\mathcal{S}_1,...,\mathcal{S}_n\subseteq\mathcal{S}$ such that $\mathcal{S}=\mathcal{S}_1\cup...\cup\mathcal{S}_n$ and, for each $j\in\{1,...,n\}$, every non-empty finite subset of $\mathcal{S}_j$ has non-empty intersection. 
\end{theorem} 

The equivalence between Theorems \ref{pq} and \ref{infinite pq} is not a new observation. The reasoning is given in the proof of Proposition 2.5 in \cite{Simon2015}. Theorem \ref{infinite pq} invites the question of what structural properties the sets $\mathcal{S}_1,...,\mathcal{S}_n$ might have. As is shown in \cite{Simon2014} and \cite{Simon2015}, using ideas from \cite{Chernikov2015}, Theorem \ref{mt version} is equivalent to the assertion that $\mathcal{S}_1,...,\mathcal{S}_n$ are first-order definable (with parameters) in the structure $(X,\mathcal{S},\in)$, this being a two-sorted structure with just one relation symbol that is interpreted as the ``element of" relation between members of $X$ and members of $\mathcal{S}$. Previously known cases of Theorem \ref{mt version} provide the same conclusion under additional assumptions on $(X,\mathcal{S})$. The case where $\text{Th}(M)$ has NIP, proved by Kaplan in \cite{Kaplan2024}, applies if every definable family of definable sets in $(X,\mathcal{S},\in)$ has finite dual VC-dimension. 

In the following result we obtain some uniformity which is not immediately implied by Theorem \ref{mt version}. For this we take inspiration from Theorem 4.1 in \cite{Kaplan2024} and its proof.

\begin{theorem}\label{definable pq}
Let $q\leq p$ be positive integers. There exist an integer $n$ and first-order formulas $\psi_1(y,w),...,\psi_n(y,w)$ such that, for any non-empty set $X$ and any non-empty set $\mathcal{S}$ of non-empty subsets of $X$, if $\text{VC}^*(X,\mathcal{S})<q$ and $(X,\mathcal{S})$ has the $(p,q)$-property then there exist tuples $d_1,...,d_n$ from the structure $(X,\mathcal{S})$ such that the following is satisfied. Each element of $\mathcal{S}$ realises $\psi_1(y,d_1)\vee...\vee\psi_n(y,d_n)$ and, for each $j\in\{1,...,n\}$, if $\mathcal{S}_j$ is the subset of $\mathcal{S}$ defined in $(X,\mathcal{S})$ by $\psi_j(y,d_j)$ then every non-empty finite subset of $\mathcal{S}_j$ has non-empty intersection.
\end{theorem} 

For finite set systems, the elements $a_1,...,a_n\in X$ in Theorem \ref{pq} can be used as parameters to define suitable $\mathcal{S}_1,...,\mathcal{S}_n$ in an obvious way (taking $\mathcal{S}_j$ to be the subset of $\mathcal{S}$ defined by the formula $a_j\in y$). However, when $\mathcal{S}$ is infinite, Theorem \ref{definable pq} says something that is not obvious. 

We now consider the history of the definable $(p,q)$-conjecture. In their Problem 29 in \cite{Chernikov2015}, Chernikov and Simon asked if the conclusion of Theorem \ref{mt version} holds when $\text{Th}(M)$ has NIP and $\varphi(x,b)$ does not fork over $M$. It was established by Chernikov and Kaplan in \cite{chernikov2012} that $\varphi(x,b)$ forks over $M$ if and only if it divides over $M$, when $\text{Th}(M)$ has NIP (or even just NTP$_2$). So ``does not fork" can be replaced by ``does not divide". 

Several authors proved special cases of a positive answer to Chernikov and Simon's question, sometimes getting stronger conclusions in those cases. In \cite{Simon2014}, Simon obtained the version where $\text{Th}(M)$ is dp-minimal and has small or medium directionality. He thanked Sergei Starchenko in that paper for motivating him to work on the conjecture. In \cite{Simon2014a}, Simon and Starchenko obtained a stronger version (in which the conclusion is that $\varphi(x,b)$ belongs to a type that is definable over $M$) when $\text{Th}(M)$ belongs to a certain class of dp-minimal theories, including all those with definable Skolem functions. In \cite{Simon2015}, Simon proved the definable $(p,q)$-conjecture when $\text{Th}(M)$ has NIP and small or medium directionality. In \cite{Boxall2018}, Boxall and Kestner proved it when $\text{Th}(M)$ is distal (a special case of NIP). In his PhD thesis \cite{Rakotonarivo2021}, Rakotonarivo proved it for dense pairs of geometric distal structures, under the assumption that algebraic closure equals definable closure in the geometric structures. In \cite{Guerrero2024}, And\'{u}jar Guerrero proved the definable $(\omega,2)$-conjecture assuming that $\varphi(x,y)$ has VC-codensity at most 2. While this is quite a strong assumption on $\varphi(x,y)$, no assumption was made about $\text{Th}(M)$. 

In \cite{Kaplan2024}, Itay Kaplan gave a full positive answer to Chernikov and Simon's question, obtaining the case of Theorem \ref{mt version} in which $\text{Th}(M)$ has NIP. His proof adapts the argument in \cite{Boxall2018} and also uses knowledge of compressibility of types in NIP theories and, crucially, a better kind of strict Morley sequence than had been used in \cite{Boxall2018}. In turn, our proof of Theorem \ref{mt version} takes Kaplan's argument further with the aid of a slightly different notion of strict Morley sequence. Theorem \ref{mt version} was stated as a conjecture in \cite{Simon2015}.

Our strict Morley sequences over a model always exist (with no assumptions on the theory) and behave well with regard to NIP formulas. Along the way, we show that an NIP formula divides over a model in an arbitrary theory if and only if it quasi-divides over that model. This may be of independent interest. 

Our proof of Theorem \ref{mt version} can be adapted to establish a family of results parameterised by certain subsets of a monster model. We obtain the following. 

\begin{theorem}\label{refined version}
Let $L$ be a language and let $M\prec N\prec U$ be $L$-structures, with $N$ and $U$ sufficiently saturated. Let $A$ be the set of realisations in $U$ of a deductively closed partial type $t(x)$ over $N$ that is definable over $M$ in the strong sense that, for every $\varphi(x,y)$, $\{b\in N:\varphi(x,b)\in t(x)\}$ is definable over $M$. Let $\varphi(x,y)$ be a first-order $L$-formula that has NIP with respect to $\text{Th}(M)$. Let $r(y)$ be a complete type over $M$. If $\{\varphi(x,b):b\in r(N)\}$ is realised in $A$ then there is an $L$-formula $\psi(y)$ over $M$ such that $\{\varphi(x,b):b\in \psi(N)\}$ is realised in $A$. 
\end{theorem}

This collapses to Theorem \ref{mt version} when $A=U^{|x|}$, given that dividing equals quasi-dividing in this setting. The motivation for generalising Theorem \ref{mt version} in this way originated in the PhD work of Tsinjo Rakotonarivo. To prove the definable $(p,q)$-conjecture for certain dense pairs of distal geometric structures, it was useful to have a refined version of the result for the underlying geometric structures. Roughly speaking, the refinement replaced the notion of not being empty with the notion of having dimension at least $d$, for some positive integer $d$. As a demonstration of the utility of the ideas behind Theorem \ref{refined version}, we include a version of Rakotonarivo's proof for dense pairs of geometric structures. For this, $A$ is taken to be a certain set of $n$-tuples with dimension at least $d$ over $N$. 

The paper is structured as follows. In Section 2 we recall some preliminary material. In Section 3 we establish the existence of what we call strictly invariant extensions of complete types over models, in arbitrary theories, and prove that their Morley sequences behave well with respect to NIP formulas. In Section 4 we prove Theorems \ref{mt version}, \ref{definable pq} and \ref{refined version}. In Section 5 we give the application of Theorem \ref{refined version} to dense pairs. 

We are grateful to Ita\"{i} Ben Yaacov for a conversation that helped us arrive at the appropriate conditions for the set $A$ in Theorem \ref{refined version}. 

\section{Notation and preliminaries}

We recall some basic definitions and facts about set systems, NIP and the model-theoretic notion of dividing. A helpful book that covers these and much more is \cite{Simon2015guide}.

Given a non-empty set $X$, a non-empty set $\mathcal{S}$ of subsets of $X$ and a subset $A\subseteq X$, we say $(X,\mathcal{S})$ shatters $A$ if $\{A\cap Y:Y\in\mathcal{S}\}=\{B:B\subseteq A\}$. The VC-dimension of $(X,\mathcal{S})$ is defined to be $$\text{VC}(X,\mathcal{S})=\max\{\left|A\right|:(X,\mathcal{S})\text{ shatters }A\}$$ if this maximum exists. Otherwise, $\text{VC}(X,\mathcal{S})=\infty$. 

The dual of $(X,\mathcal{S})$ is $(\mathcal{S},\{\mathcal{S}_x:x\in X\})$ where, for each $x\in X$, $\mathcal{S}_x=\{Y\in\mathcal{S}:x\in Y\}$. The dual VC-dimension of $(X,\mathcal{S})$ is defined to be $$\text{VC}^*(X,\mathcal{S})=\text{VC}(\mathcal{S},\{\mathcal{S}_x:x\in X\}).$$ It is known that $\text{VC}(X,\mathcal{S})$ is finite if and only if $\text{VC}^*(X,\mathcal{S})$ is finite.

Throughout this paper, $L$ is a (possibly many-sorted) language and $T$ is a complete first-order $L$-theory. In general, no further assumptions are made about $T$.  For $L$-structures $M$ and $N$, we use the notation $M\prec^* N$ to mean that $N$ is a $\kappa$-saturated and strongly $\kappa$-homogeneous elementary extension of $M$, for some sufficiently large $\kappa$. If $x$ is a tuple of variables then $M^{|x|}$ is the cartesian power of sorts of $M$ determined by the sorts of the variables. We will not always indicate which cartesian power of sorts is being used, and so could simply write $a\in M$ instead of $a\in M^{|x|}$. 

Given a first-order $L$-formula $\varphi$ and disjoint tuples of variables $x$ and $y$, such that every free variable of $\varphi$ appears in $x$ or $y$, the partitioned formula $\varphi(x,y)$ is said to have NIP (with respect to $T$) if, for every model $M\models T$, the set system $(M^{\left|x\right|},\{\varphi(M,b):b\in M^{\left|y\right|}\})$ has finite VC-dimension, and so also finite dual VC-dimension. The theory $T$ is said to have NIP if every partitioned first-order $L$-formula has NIP with respect to $T$.

Let $M\prec^*N\models T$. Let $\varphi(x,y)$ be an $L$-formula and $b\in N^{\left|y\right|}$. We say $\varphi(x,b)$ divides over $M$ if there is a sequence $(b_n)_{n\in\mathbb{N}}$ of realisations of $\text{tp}(b/M)$ in $N$ that is indiscernible over $M$ and such that $\{\varphi(x,b_n):n\in\mathbb{N}\}$ is inconsistent. We say $\varphi(x,b)$ quasi-divides over $M$ if $\{\varphi(x,b'):b'\models\text{tp}(b/M)\}$ is not finitely consistent.

\section{Strictly invariant extensions and strict Morley sequences}

The concept of a strictly invariant extension played an important role in \cite{chernikov2012}, where the existence of such extensions was established, under certain conditions. Here we define a slightly modified version of the notion of invariant extension, and we consider only the case of types over a model. 

\begin{definition}\label{strict invariance}
Let $M\prec^* N\models T$ and let $r$ be a complete type over $M$. A strictly invariant extension of $r$ to $N$ is a complete type $s$ over $N$ such that

\begin{enumerate}
\item $r\subseteq s$,
\item $s$ is invariant over $M$ and
\item for every type $t(y_1,...,y_n)$ over $M$, if $t(y_1,...,y_n)\cup r(y_1)\cup ... \cup r(y_n)$ is consistent then so is $t(y_1,...,y_n)\cup s(y_1)\cup ... \cup s(y_n)$.
\end{enumerate}
\end{definition}

To compare this with Definition 3.8 in \cite{chernikov2012}, we remark that item (3) of our Definition \ref{strict invariance} is equivalent to saying that if $b$ realises $s$ then, for any $\psi(y,c)\in s$, where the formula $\psi(y,z)$ is over $M$, $\psi(b,z)$ does not quasi-divide over $M$. Since forking and quasi-dividing are equivalent over models in NTP$_2$ theories, as is shown in \cite{chernikov2012}, the two notions of strictly invariant extension agree, for global extensions of types over models, in the NTP$_2$ setting.

We show that strictly invariant extensions of types over models (as in Definition \ref{strict invariance}) always exist and that, moreover, we can obtain one that is a coheir. This is in the spirit of Proposition 3.7 in \cite{chernikov2012}, but it is not covered by that result and the proof is different. 

\begin{proposition}\label{strictinv}
Let $M\prec^* N\models T$ and let $r$ be a complete type over $M$. There exists a strictly invariant extension $s$ of $r$ to $N$ such that $s$ is finitely realised in $M$. 
\end{proposition}

\proof  Let $U$ be a further sufficiently saturated elementary extension, so $M\prec^* N\prec^* U$. Let $(b_i)_{i<\lambda}$ be an infinite tuple of realisations of $r$ in $N$ such that, for every $n\in\mathbb{N}$ and type $t(y_1,...,y_n)$ over $M$, if $t(y_1,...,y_n)\cup r(y_1)\cup ... \cup r(y_n)$ is consistent then there exist $i_1,...,i_n<\lambda$ such that $(b_{i_1},...,b_{i_n})\models t$. Let $u((y_i)_{i<\lambda})=\text{tp}((b_i)_{i<\lambda}/M)$. Let $$v(y)=\{\varphi(y)\ |\ \varphi\text{ is over }N\text{ and }\neg\varphi\text{ is not realised in }M\}.$$ We complete the proof by showing that $$u((y_i)_{i<\lambda})\cup v(y_0)\cup\{\varphi(y_i)\leftrightarrow\varphi(y_j)\ |\ i,j<\lambda\text{ and }\varphi\text{ is over }N\}$$ is consistent, since then, for any realisation $(b'_i)_{i<\lambda}$, we could choose $s=\text{tp}(b'_0/N)$. Let $\theta(y_{i_1},...,y_{i_m})\in u$ and, for simplicity of notation, assume $(i_1,...,i_m)=(1,...,m)$. Let $$w(y)=\{\neg\varphi(y)\ |\ \varphi\text{ is over }N\text{ and }\{\theta(y_1,...,y_m),\varphi(y_1),...,\varphi(y_m)\}\text{ is inconsistent}\}.$$

{\bf Claim:} The set $v(y)\cup w(y)$ is consistent.

{\bf Proof of Claim:} Suppose not. Then, by compactness, there exist $$\neg\varphi_1(y,c_1),...,\neg\varphi_k(y,c_k)\in w(y)$$ such that, for every $b\in U$, if $b\models v(y)$ then $b\models\varphi_1(y,c_1)\vee...\vee\varphi_k(y,c_k)$. We have made the parameters $c_1,...,c_k\in N$ explicit so that $\varphi_1(y,z_1)$,...,$\varphi_k(y,z_k)$ are over $M$. We consider two cases.

Case 1: Assume that, for every $\alpha(z_1,...,z_k)\in \text{tp}(c_1,...,c_k/M)$, there exists $(c'_1,...,c'_k)\in\alpha(M)$ such that, for some $i\in \{1,...,k\}$, $\varphi_i(y,c'_i)\in r$. For each such $(c'_1,...,c'_k)\in\alpha(M)$, $\{\theta(y_1,...,y_m),\varphi_i(y_1,c'_i),...,\varphi_i(y_m,c'_i)\}$ is consistent, for some $i\in \{1,...,k\}$. By compactness, we have $i\in\{1,...,k\}$ and $c'_i\models \text{tp}(c_i/M)$ such that $\{\theta(y_1,...,y_m),\varphi_i(y_1,c'_i),...,\varphi_i(y_m,c'_i)\}$ is consistent. But then $\{\theta(y_1,...,y_m),\varphi_i(y_1,c_i),...,\varphi_i(y_m,c_i)\}$ is consistent, contradicting the fact that $\neg\varphi_i(y,c_i)\in w(y)$. 

Case 2: Assume there exists $\alpha(z_1,...,z_k)\in \text{tp}(c_1,...,c_k/M)$ such that, for all $(c'_1,...,c'_k)\in\alpha(M)$ and $i\in\{1,...,k\}$, $\varphi_i(y,c'_i)\notin r$. Then, for each $\alpha(z_1,...,z_k)\in \text{tp}(c_1,...,c_k/M)$, there exist $b,c'_1,...,c'_k\in M$ such that $b\models v(y)$, $(c'_1,...,c'_k)\models \alpha(z_1,...,z_k)$ and $b$ does not realise $\varphi_1(y,c'_1)\vee...\vee\varphi_k(y,c'_k)$. By compactness, there exist realisations $b\models v(y)$ and $(c'_1,...,c'_k)\models \text{tp}(c_1,...,c_k/M)$ such that $b$ does not realise $\varphi_1(y,c'_1)\vee...\vee\varphi_k(y,c'_k)$. Applying an automorphism, we can replace $(c'_1,...,c'_k)$ with $(c_1,...,c_k)$ and so obtain a contradiction.

As Cases 1 and 2 are exhaustive, we have a contradiction. This completes the proof of the claim.

Therefore there exists a realisation $b\models v(y)$ such that, for all $\varphi(y)\in \text{tp}(b/N)$, $$\{\theta(y_1,...,y_m),\varphi(y_1),...,\varphi(y_m)\}$$ is consistent. So then there exist $b_1,...,b_m\models \text{tp}(b/N)$ such that $(b_1,...,b_m)\models\theta(y_1,...,y_m)$.

It follows that $$\{\theta(y_1,...,y_m)\}\cup v(y_1)\cup\{\varphi(y_i)\leftrightarrow\varphi(y_j)\ |\ i,j=1,...,m\text{ and }\varphi\text{ is over }N\}$$ is consistent. By compactness then $$u((y_i)_{i<\lambda})\cup v(y_0)\cup\{\varphi(y_i)\leftrightarrow\varphi(y_j)\ |\ i,j<\lambda\text{ and }\varphi\text{ is over }N\}$$ is consistent, as desired. \endproof

\begin{definition}
Let $M\models T$ and let $r$ be a complete type over $M$. A strict Morley sequence of realisations of $r$ is a Morley sequence over $M$ for a strictly invariant extension $s$ of $r$ to some $N$ such that $M\prec^* N$. That is to say, it is a sequence $(b_n)_{n<\omega}$ such that, for each $n$, $b_n$ realises the restriction of $s$ to $M\cup\{b_0,...,b_{n-1}\}$. 
\end{definition}

This use of the phrase ``strict Morley sequence" is new only in that we are using the notion of strictly invariant extension given in Definition \ref{strict invariance}. An important property of the strict Morley sequences considered in \cite{chernikov2012}, in the NTP$_2$ setting, is that they witness dividing over a model. For our strict Morley sequences, the same is true of NIP formulas in an arbitrary theory. When proving this, we also relate dividing to quasi-dividing. Here, and subsequently, we use the well-known fact that NIP formulas have finite alternation number. 

\begin{proposition}\label{chardividing}
Let $M\prec^* N\models T$ and $b\in N$. Let $\varphi(x,y)$ be an NIP formula, with parameters in $M$. Let $(b_n)_{n\in\mathbb{N}}$ be a strict Morley sequence of realisations of $\text{tp}(b/M)$. The following are equivalent.

\begin{enumerate}\label{dividing}
\item $\varphi(x,b)$ divides over $M$.
\item $\{\varphi(x,b_n):n\in\mathbb{N}\}$ is inconsistent.
\item $\varphi(x,b)$ quasi-divides over $M$.
\end{enumerate}
\end{proposition}

\proof From the definitions and the indiscernibility of $(b_n)_{n\in\mathbb{N}}$ over $M$, $(2)\Rightarrow (1)\Rightarrow (3)$. We prove $(3)\Rightarrow (2)$. Suppose (2) is false and so $\{\varphi(x,b_n):n\in\mathbb{N}\}$ is consistent. Let $a\in N^{|x|}$ be such that $N\models\varphi(a,b_n)$ for large enough $n$ and, subject to that constraint, the truth value of $\varphi(a,b_n)$ alternates the maximum number of times. We may assume there is some $m\in\mathbb{N}$ such that, for $n=1,...,m-1$, $N\models\varphi(a,b_n)\leftrightarrow\neg\varphi(a,b_{n+1})$ and, for $n\geq m$, $N\models\varphi(a,b_n)$. Let $b'_1,...,b'_k\models\text{tp}(b/M)$. Let $s$ be the strictly invariant extension of $\text{tp}(b/M)$ to $N$ used to construct the Morley sequence $(b_n)_{n\in\mathbb{N}}$. We then have $(b''_1,...,b''_k)\models \text{tp}(b'_1,...,b'_k/M)$ such that, for each $i=1,...,k$, $b''_i$ realises the restriction of $s$ to $M\cup\{b_i:i\in\mathbb{N}\}$. We can then obtain a Morley sequence $(b_{\omega+i})_{i\in\mathbb{N}}$ for $s$ over $M\cup\{b_i:i\in\mathbb{N}\}\cup\{b''_1,...,b''_k\}$. 

By compactness, we may assume $N\models\varphi(a,b_{\omega+i})$ for all $i\in\mathbb{N}$. For each $i=1,...,k$, if we set $b_\omega=b''_i$ then the sequence $(b_1,b_2,b_3,...,b_\omega,b_{\omega+1},b_{\omega+2},...)$ is indiscernible over $M$. We must then have $N\models\varphi(a,b''_i)$, since otherwise we could rechoose $a$ so as to contradict the maximality used in the original choice of $a$. So $N\models\varphi(a,b''_1)\wedge...\wedge\varphi(a,b''_k)$ and so $\{\varphi(x,b''_1),...,\varphi(x,b''_k)\}$ is consistent. Therefore $\{\varphi(x,b'_1),...,\varphi(x,b'_k)\}$ is consistent. We have shown that $\varphi(x,b)$ does not quasi-divide over $M$.\endproof

We shall have use for the following refinement of Proposition \ref{dividing}. The proof is an easy modification of the previous one. 

\begin{proposition}\label{refined equivalence}
Let $M\prec^* N \prec^* U\models T$ and $b\in N$. Let $\varphi(x,y)$ be an NIP formula, with parameters in $M$. Let $(b_n)_{n\in\mathbb{N}}$ be a strict Morley sequence of realisations of $\text{tp}(b/M)$. Let $A\subseteq U^{|x|}$ be type-definable over $N$ by means of a partial type that is invariant over $M$. The following are equivalent, where $r=\text{tp}(b/M)$.
\begin{enumerate}
\item For some indiscernible sequence $(b'_n)_{n\in\mathbb{N}}$ of realisations of $r$, $\{\varphi(x,b'_n):n\in\mathbb{N}\}$ is not realised in $A$.
\item $\{\varphi(x,b_n):n\in\mathbb{N}\}$ is not realised in $A$.
\item $\{\varphi(x,b'):b'\in r(N)\}$ is not realised in $A$.
\end{enumerate}
\end{proposition}

\section{Definable $(p,q)$-theorems}

In this section we prove Theorems \ref{mt version}, \ref{definable pq} and \ref{refined version}. We shall need two preliminary facts. The following is well-known and a proof is given in \cite{Simon2015} (Corollary 2.3 in that paper). 

\begin{fact}\label{fact 1} 
    Let $T$ be a complete theory and suppose $\theta(x,z)$ has NIP with respect to $T$. Let $M\prec^* N\models T$ and $c\in N$. Suppose $\theta(x,c)$ does not divide over $M$. Then there is a formula $\tau\in\text{tp}(c/M)$ such that, for every $c'\in \tau(N)$, $\theta(x,c')$ does not divide over $M$.
\end{fact}

The following is part of Lemma 2.4 in \cite{Simon2015}, modified to include the assumption, in item (2), that $\theta(N,c)\neq\emptyset$ for every $c\in\tau(N)$. Kaplan points out in \cite{Kaplan2024} that this assumption is necessary. It clearly follows from item (1). 

\begin{fact}\label{fact 2}
    Let $T$ be a complete theory and $M\prec^*N\models T$. Suppose $\theta(x,z)$ has NIP with respect to $T$ and $\tau(z)$ is a formula over $M$. The following are equivalent.
    \begin{enumerate}
        \item For all $c\in\tau(N)$, $\theta(x,c)$ does not divide over $M$. 
        \item The set system $(N^{|x|},\{\theta(N,c):c\in\tau(N)\})$ has the $(p,q)$-property, for some $p\geq q>\text{VC}^*(N^{|x|},\{\theta(N,c):c\in\tau(N)\})$, and, for every $c\in\tau(N)$, $\theta(N,c)\neq\emptyset$.
    \end{enumerate}
\end{fact}

\subsection{Proof of Theorem \ref{mt version}} Let $T$ be a complete $L$-theory and let $M\prec^* N\prec^* U\models T$. Let $\varphi(x,y)$ be a first-order $L$-formula that has NIP with respect to $T$. Let $b\in N$ and suppose $\varphi(x,b)$ does not divide over $M$. 

Using Proposition \ref{strictinv}, let $s(y)$ be a strictly invariant extension of $\text{tp}(b/M)$ to $N$ such that $s$ is finitely realised in $M$. Let $(b_n)_{n\in\mathbb{N}}$ be a Morley sequence for $s$ over $M$ in $N$. Let $(b'_n)_{n\in\mathbb{N}}$ be an indiscernible sequence in $N$ over $M$ that has the same EM-types over $M$ as the sequence $(...,b_3,b_2,b_1)$. Since $\varphi(x,b)$ does not divide over $M$, the set $\{\varphi(x,b'_n):n\in\mathbb{N}\}$ is consistent. 

Let $a\in U^{\left|x\right|}$ be such that the truth value of $\varphi(a,b'_n)$ alternates as many times as possible, subject to the constraint that $U\models\varphi(a,b'_n)$ eventually. We may assume there exists $m\in\mathbb{N}$ such that $U\models\varphi(a,b'_n)\leftrightarrow\neg\varphi(a,b'_{n+1})$, for all $n\in\{1,...,m-1\}$, and $U\models\varphi(a,b'_n)$, for all $n\geq m$. By compactness, there is some $k$ such that

\begin{enumerate}

\item there exists $a\in U^{\left|x\right|}$ such that $$U\models\bigwedge\limits_{n\in\{1,...,m-1\}}(\varphi(a,b'_n)\leftrightarrow\neg\varphi(a,b'_{n+1}))\wedge\varphi(a,b'_m)\wedge\varphi(a,b'_{m+2})\wedge...\wedge\varphi(a,b'_{m+k+1}),$$

\item for all $a$ as in (1), $U\models\varphi(a,b'_{m+1})$.

\end{enumerate}

Then, for the Morley sequence $(b_n)_{n\in\mathbb{N}}$, we have that

\begin{enumerate}

\item[(3)] there exists $a\in U^{\left|x\right|}$ such that $$U\models\varphi(a,b_1)\wedge...\wedge\varphi(a,b_k)\wedge\varphi(a,b_{k+2})\wedge\bigwedge\limits_{n\in\{k+2,...,k+m\}}(\varphi(a,b_n)\leftrightarrow\neg\varphi(a,b_{n+1})),$$

\item[(4)] for all $a$ as in (3), $U\models\varphi(a,b_{k+1})$.

\end{enumerate}

Since $s$ is finitely realised in $M$, there exist $d_{k+2},....,d_{k+m+1}\in M$ such that 

\begin{enumerate}

\item[(5)] there exists $a\in U^{\left|x\right|}$ such that $$U\models\varphi(a,b_1)\wedge...\wedge\varphi(a,b_k)\wedge\varphi(a,d_{k+2})\wedge\bigwedge\limits_{n\in\{k+2,...,k+m\}}(\varphi(a,d_n)\leftrightarrow\neg\varphi(a,d_{n+1})),$$

\item[(6)] for all $a$ as in (5), $U\models\varphi(a,b_{k+1})$.

\end{enumerate}

Let $z=(y_1,...,y_k)$ and let $\theta(x,z)$ be the formula $$\varphi(x,y_1)\wedge...\wedge\varphi(x,y_k)\wedge\varphi(x,d_{k+2})\wedge\bigwedge\limits_{n\in\{k+2,...,k+m\}}(\varphi(x,d_n)\leftrightarrow\neg\varphi(x,d_{n+1})).$$ By (5) and (6), there exists $a\in U^{\left|x\right|}$ such that $U\models\theta(a,b_1,...,b_k)$ and, for any such $a$, $U\models\varphi(a,b_{k+1})$. Let $\varphi'(x,y)$ be the formula $$\varphi(x,y)\wedge\varphi(x,d_{k+2})\wedge\bigwedge\limits_{n\in\{k+2,...,k+m\}}(\varphi(x,d_n)\leftrightarrow\neg\varphi(x,d_{n+1})).$$ We note that $\varphi'(x,y)$ has NIP. By (5), (6) and the indiscernibility of $(b_n)_{n\in\mathbb{N}}$ over $M$, $\{\varphi'(x,b_n):n\in\mathbb{N}\}$ is consistent. It follows, by Proposition \ref{chardividing}, that $\varphi'(x,b)$ does not quasi-divide over $M$. Then $\theta(x,b_1,...,b_k)$ does not quasi-divide over $M$ and so does not divide over $M$. Furthermore, $\theta(x,z)$ has NIP with respect to $T$ since $\varphi(x,y)$ does.

By Fact \ref{fact 1}, there is a formula $\tau(z)\in \text{tp}(b_1,...,b_k/M)$ such that, for all $(b''_1,...,b''_k)\models\tau$, $\theta(x,b''_1,...,b''_k)$ does not divide over $M$. Then, by Fact \ref{fact 2}, the set system $$(U^{|x|},\{\theta(U,b''_1,...,b''_k):(b''_1,...,b''_k)\in\tau(N)\})$$ has the $(p',q')$-property for some $$p'\geq q'> \text{VC}^*(U^{|x|},\{\theta(U,b''_1,...,b''_k):(b''_1,...,b''_k)\in\tau(N)\}).$$ By Theorem \ref{infinite pq} and the fact that $N\prec^*U$, there is a positive integer $n'$, depending only on $(p',q')$, such that there exist $a_1,...,a_{n'}\in U^{\left|x\right|}$ such that, for all $c\in \tau(N)$, $U\models\theta(a_i,c)$ for some $i\in\{1,...,n'\}$. 

Using the fact that $(b_n)_{n\in\mathbb{N}}$ is a strict Morley sequence over $M$, for any finite set $B\subseteq N^{|y|}$ of realisations of $\text{tp}(b/M)$ there is an automorphism $\sigma$ of $N$ over $M$ that maps $B$ to a set of realisations of $\text{tp}(b_{k+1}/Mb_1...b_k)$. Then, for all $a\in U^{|x|}$ such that $U\models \theta(a,b_1,...,b_k)$, we have $U\models \varphi(a,b')$ for all $b'\in \sigma(B)$. It follows that, for all $a\in U^{|x|}$ such that $U\models\theta(a,\sigma^{-1}(b_1),...,\sigma^{-1}(b_k))$, $U\models \varphi(a,b')$ for all $b'\in B$.

By compactness, there is a formula $\psi\in \text{tp}(b/M)$ such that, for any $b''_1,...,b''_{n'}\in \psi(N)$, there exists $c\in \tau(N)$ such that, for all $a\in U^{|x|}$, if $U\models \theta(a,c)$ then, for all $i\in \{1,...,n'\}$, $U\models\varphi(a,b''_i)$. If for each $i\in\{1,...,n'\}$ there exists $b''_i\in\psi(N)$ such that $U\models\neg\varphi(a_i,b''_i)$ then we obtain a contradiction. Therefore, for at least one $i\in\{1,...,n'\}$, $U\models\varphi(a_i,b'')$ for all $b''\in\psi(N)$. This completes the proof of Theorem \ref{mt version}.

\subsection{Proof of Theorem \ref{definable pq}}
Let $L$ be a two-sorted language with one binary relation symbol $\in$, whose first coordinate corresponds to the first sort and whose second coordinate corresponds to the second sort. When interpreting a set system $(X,\mathcal{S})$ as an $L$-structure, $\in$ is given its usual meaning, ``is an element of". We shall first prove a weaker version of Theorem \ref{definable pq} from which it will be easy to deduce the full result.

\begin{proposition}\label{weak definable pq}
Let $q\leq p$ be positive integers. There exist positive integers $k, n_1,...,n_k$ and first-order $L$-formulas $\psi_{i,j}(y,w)$, for $i\in\{1,...,k\}$ and $j\in \{1,...,n_i\}$, such that, for any non-empty set $X$ and any non-empty set $\mathcal{S}$ of non-empty subsets of $X$, if $\text{VC}^*(X,\mathcal{S})<q$ and $(X,\mathcal{S})$ has the $(p,q)$-property then there exist $i\in \{1,...,k\}$ and tuples $d_1,...,d_{n_i}$ from the structure $(X,\mathcal{S})$ such that the following is satisfied. Each element of $\mathcal{S}$ realises $\psi_{i,1}(y,d_1)\vee...\vee\psi_{i,n_i}(y,d_{n_i})$ and, for each $j\in\{1,...,n_i\}$, if $\mathcal{S}_j$ is the subset of $\mathcal{S}$ defined in $(X,\mathcal{S})$ by $\psi_{i,j}(y,d_j)$ then every non-empty finite subset of $\mathcal{S}_j$ has non-empty intersection.
\end{proposition} 

\proof Suppose not. Using Theorem \ref{mt version} and the fact that the set of all finite sequences of $L$-formulas is countable, there exist positive integers $q\leq p$ and a family of set systems $\{N_i=(X_i,\mathcal{S}_i):i\in\mathbb{N}\}$ satisfying the following conditions.

\begin{enumerate}
    \item For each $i\in\mathbb{N}$, $X_i\neq\emptyset$, $\mathcal{S}_i\neq\emptyset$, each element of $\mathcal{S}_i$ is non-empty, $(X_i,\mathcal{S}_i)$ has the $(p,q)$-property and $\text{VC}^*(X_i,\mathcal{S}_i)<q$.
    \item For each finite sequence of $L$-formulas $\psi_1(y,w),...,\psi_n(y,w)$, there is an integer $m$ such that, for all $i\geq m$, there do not exist parameters $d_1,...,d_n$ from $(X_i,\mathcal{S}_i)$ such that both \begin{enumerate} \item $\mathcal{S}_i=\psi_1(N_i,d_1)\cup...\cup\psi_n(N_i,d_n)$ and \item for each $j\in\{1,...,n\}$, every non-empty finite subset of $$\{b\in \mathcal{S}_i:N_i\models\psi_j(b,d_j)\}$$ has non-empty intersection. \end{enumerate}
\end{enumerate}

Let $\mathcal{U}$ be a non-principal ultrafilter on $\mathbb{N}$ and let $(X,\mathcal{S})$ be the ultraproduct of $(X_i,\mathcal{S}_i)_{i\in\mathbb{N}}$ with respect to $\mathcal{U}$. Let $\varphi(x,y)$ be the $L$-formula $x\in y$ and let $N=(X,\mathcal{S})$. By \L o\'s's lemma, $X\neq\emptyset$, $\mathcal{S}\neq\emptyset$, $\varphi(N,b)$ is non-empty for each $b\in\mathcal{S}$, the set system $(X,\{\varphi(N,b):b\in\mathcal{S}\})$ has the $(p,q)$-property and $\text{VC}^*(X,\{\varphi(N,b):b\in\mathcal{S}\})<q$. 

Let $M\prec N$ such that $M$ is countable. By Fact \ref{fact 2}, for each $b\in\mathcal{S}$, $\varphi(x,b)$ does not divide over $M$. For each $b\in\mathcal{S}$, let $\psi_b(y,d)$, $\theta_b(x,z,d)$ and $\tau_b(z,d)$ be the formulas $\psi(y),\theta(x,z)$ and $\tau(z)$ from the proof of Theorem \ref{mt version}. We have made the parameters $d\in M$ explicit so that $\psi_b(y,w), \theta_b(x,z,w)$ and $\tau_b(z,w)$ have no parameters. Let $p'_b,q'_b$ and $n'_b$ be the numbers $p', q'$ and $n'$ from the proof of Theorem \ref{mt version}, respectively. By the $\aleph_1$-saturation of the ultraproduct $(X,\mathcal{S})$, there is a finite $B\subseteq\mathcal{S}$ such that $\mathcal{S}=\bigcup\limits_{b\in B}\psi_b(N)$. While the sort of the parameter variable $w$ depends on $b$, we may assume that it remains constant as $b$ ranges over the finite set $B$.

By \L o\'s's lemma and reasoning given in the proof of Theorem \ref{mt version}, there is some $V\in\mathcal{U}$ such that, for each $i\in V$ and $b\in B$, there exists $d_b\in N_i$ such that

\begin{enumerate} 
\item $p_b'\geq q_b'> \text{VC}^*(X_i,\{\theta_b(N_i,c,d_b):c\in\tau_b(N_i,d_b)\})$,
\item for each $c\in\tau_b(N_i,d_b)$, $\theta_b(N_i,c,d_b)\neq\emptyset$,
\item $(X_i,\{\theta_b(N_i,c,d_b):c\in\tau_b(N_i,d_b)\})$ has the $(p_b',q_b')$-property,
\item for any $b_1'',...,b_{n_b'}''\in \psi_b(N_i,d_b)$ there exists $c\in \tau_b(N_i,d_b)$ such that, for all $a\in X_i$, if $N_i\models\theta_b(a,c,d_b)$ then $N_i\models \varphi(a,b_1'')\wedge ...\wedge \varphi(a, b_{n_b'}'')$, and
\item $\mathcal{S}_i=\bigcup\limits_{b\in B}\psi_b(N_i,d_b)$.
\end{enumerate} 

For each $i\in V$ and $b\in B$, we obtain $a_1,...,a_{n_b'}$ in an elementary extension $U_i$ of $N_i$ such that, for each $c\in \tau_b(N_i,d_b)$, $U_i\models \theta_b(a_1,c,d_b)\vee ...\vee\theta_b(a_{n_b'},c,d_b)$. As in the proof of Theorem \ref{mt version}, it follows that, for some $j\in\{1,...,n_b'\}$, $U_i\models \varphi(a_j,b'')$ for all $b''\in \psi_b(N_i,d_b)$. So, for each $i\in V$ and $b\in B$, we have that every non-empty finite subset of  $\{\varphi(N_i,b''):b''\in\psi_b(N_i,d_b)\}$ has non-empty intersection. Since $V$ is infinite, this is a contradiction. \endproof

Theorem \ref{definable pq} is the strengthening of Proposition \ref{weak definable pq} in which $k=1$. We obtain this by using, as the formulas $\psi_1,...,\psi_n$ in the statement of Theorem \ref{definable pq}, conjunctions of formulas from the statement of Proposition \ref{weak definable pq}. In particular, we use all conjunctions of the form $\psi_{1,j_1}\wedge...\wedge\psi_{k,j_k}$, where, for each $i\in\{1,...,k\}$, $j_i\in\{1,...,n_i\}$. 

\subsection{Proof of Theorem \ref{refined version}} Let $T$ be a complete $L$-theory and let $M\prec^* N\prec^* U\models T$. Let $t(x)$ be a deductively closed partial type over $N$ such that, for every $\varphi(x,y)$, $\{b\in N:\varphi(x,b)\in t(x)\}$ is definable over $M$. As in the statement of Theorem \ref{refined version}, let $A\subseteq U^{|x|}$ be the set of realisations of $t(x)$. 

\begin{lemma}\label{A lemma}
    For every formula $\varphi(x,y)$ over $M$, there is a formula $\psi(y)$ over $M$ such that, for all $b\in N$, $\varphi(U,b)\cap A\neq\emptyset$ if and only if $b\in \psi(N)$. 
\end{lemma} 

\proof By the assumptions on $t(x)$, there is a formula $\psi'(y)$ over $M$ such that, for all $b\in N$, $\neg\varphi(x,b)\in t(x)$ if and only if $b\in\psi'(N)$. Since $t(x)$ is deductively closed, we then have that $A\subseteq U^{|x|}\setminus\varphi(U,b)$ if and only if $b\in \psi'(N)$. It follows that $\varphi(U,b)\cap A\neq\emptyset$ if and only if $b\in N^{|y|}\setminus\psi'(N)$. So we may choose $\psi$ to be $\neg\psi'$. \endproof

Using this, together with modified versions of Facts \ref{fact 1} and \ref{fact 2}, it is easy to adapt the proof of Theorem \ref{mt version} to obtain a proof of Theorem \ref{refined version}. The modified versions of the facts are as follows. 

\begin{proposition}\label{fact 3}
    Let $\theta(x,z)$ have NIP and let $c\in N^{|z|}$ have the property that, for every sequence $(c_n)_{n\in\mathbb{N}}$ of realisations of $\text{tp}(c/M)$ in $N$ that is indiscernible over $M$, $\{\theta(x,c_n):n\in\mathbb{N}\}$ is realised in $A$. Then there is a formula $\tau\in \text{tp}(c/M)$ such that, for every sequence $(c_n)_{n\in\mathbb{N}}$ in $\tau(N)$ that is indiscernible over $M$, $\{\theta(x,c_n):n\in\mathbb{N}\}$ is realised in $A$.
\end{proposition}

\proof It is known that, for an indiscernible sequence $(c_n)_{n\in\mathbb{N}}$ in $N$, $\{\theta(x,c_n):n\in\mathbb{N}\}$ is consistent if and only if $\{\theta(x,c_n):n=1,...,q\}$ is consistent, where $q$ is any number greater than $\text{VC}^*(U^{|x|},\{\theta(U,c):c\in N^{|z|}\})$. A proof is given in \cite{Simon2015} (Lemma 2.2 there). The same reasoning can be used to show that $\{\theta(x,c_n):n\in\mathbb{N}\}$ is realised in $A$ if and only if $\{\theta(x,c_n):n=1,...,q\}$ is realised in $A$, where the sequence $(c_n)_{n\in\mathbb{N}}$ in $N$ is assumed to be indiscernible over $M$, except that we use an automorphism that fixes $N$ setwise and $M$ pointwise and appeal to the fact that $t(x)$ is invariant over $M$.

We can then adapt the proof of Corollary 2.3 in \cite{Simon2015} as follows. Let $j(z_1,...,z_q)$ be the partial type over $M$ asserting that any realisation is indiscernible over $M$. By Lemma \ref{A lemma} and compectness, there are formulas $\beta(z_1,...,z_q)\in j$ and $\alpha(z)\in\text{tp}(c/M)$ such that, for each $(c_1,...,c_q)\in N^{q|z|}$ realising $\beta(z_1,...,z_q)\wedge\alpha(z_1)\wedge...\wedge \alpha(z_q)$, $\{\theta(x,c_n):n=1,...,q\}$ is realised in $A$. We may then choose $\tau(z)$ to be $\alpha(z)$. Then, for any sequence $(c_n)_{n\in\mathbb{N}}$ in $\tau(N)$ that is indiscernible over $M$, $\{\theta(x,c_n):n\in\mathbb{N}\}$ is realised in $A$. \endproof

\begin{proposition}\label{fact 4} 
    Let $T$ be a complete theory and $M\prec^*N\models T$. Suppose $\theta(x,z)$ has NIP with respect to $T$ and $\tau(z)$ is a formula over $M$. The following are equivalent.
    \begin{enumerate}
        \item For every sequence $(c_n)_{n\in\mathbb{N}}$ in $\tau(N)$ that is indiscernible over $M$, $\{\theta(x,c_n):n\in\mathbb{N}\}$ is realised in $A$.
        \item The set system $(A,\{\theta(U,c):c\in\tau(N)\})$ has the $(p,q)$-property, for some $p\geq q>\text{VC}^*(A,\{\theta(U,c):c\in\tau(N)\})$, and, for every $c\in\tau(N)$, $\theta(U,c)\cap A\neq\emptyset$.
    \end{enumerate}
\end{proposition}

\proof We adapt the proof of Lemma 2.4 in \cite{Simon2015}. 

Assume (2) is false. We may assume $\theta(U,c)\cap A\neq\emptyset$ for all $c\in\tau(N)$. Let $q>\text{VC}^*(A,\{\theta(U,c):c\in\tau(N)\})$. By Lemma \ref{A lemma}, there is a formula $\gamma(z_1,...,z_q)$ over $M$ such that, for all $c_1,...,c_q\in \tau(N)$, $\theta(x,c_1)\wedge...\wedge\theta(x,c_q)$ is realised in $A$ if and only if $(c_1,...,c_q)\in \gamma(N)$. For each $p\geq q$, $(A,\{\theta(U,c):c\in\tau(N)\})$ does not have the $(p,q)$-property. Therefore there are subsets of $\tau(N)$ of arbitrary finite size such that each $q$-tuple of distinct elements from such a subset realises $\neg\gamma$. By a standard application of Ramsey's theorem and compactness, there is a sequence $(c_n)_{n\in\mathbb{N}}$ in $\tau(N)$ that is indiscernible over $M$ and such that $(c_1,...,c_q)\in (\neg\gamma)(N)$ and so $\theta(z,c_1)\wedge...\wedge\theta(x,c_q)$ is not realised in $A$. Therefore (1) is false.

Now assume (1) is false. Let $(c_n)_{n\in\mathbb{N}}$ be a sequence in $\tau(N)$ that is indiscernible over $M$ and such that $\{\theta(z,c_n):n\in\mathbb{N}\}$ is not realised in $A$. We may assume $\theta(U,c_n)\neq\theta(U,c_m)$ whenever $n\neq m$, since otherwise all the sets in $\{\theta(U,c_n):n\in\mathbb{N}\}$ are the same and the second clause of (2) is immediately false. Let $q>\text{VC}^*(A,\{\theta(U,c):c\in\tau(N)\})$. As in the proof of Proposition \ref{fact 3}, it follows that $\theta(x,c_1)\wedge...\wedge\theta(x,c_q)$ is not realised in $A$. Since $t(x)$ is invariant over $M$, no $q$-element subset of $\{\theta(x,c_n):n\in\mathbb{N}\}$ is realised in $A$. Therefore $(A,\{\theta(U,c):c\in\tau(N)\})$ cannot have the $(p,q)$-property, for any $p\geq q>\text{VC}^*(A,\{\theta(U,c):c\in\tau(N)\})$. So (2) is false. \endproof

The rest of the proof of Theorem \ref{refined version} is almost exactly copied from the proof of Theorem \ref{mt version}. It is only the justification of (5) and (6) that requires a significant amount of new reasoning. 

Let $\varphi(x,y)$ be a first-order $L$-formula that has NIP with respect to $T$. Let $b\in N$ and suppose $\{\varphi(x,b'):b'\in N\text{ and }b'\models\text{tp}(b/M)\}\cap A\neq\emptyset$. Let $s(y)$ be a strictly invariant extension of $\text{tp}(b/M)$ to $N$ such that $s$ is finitely realised in $M$. Let $(b_n)_{n\in\mathbb{N}}$ be a Morley sequence for $s$ over $M$ in $N$. Let $(b'_n)_{n\in\mathbb{N}}$ be an indiscernible sequence in $N$ over $M$ which has the same EM-types over $M$ as the sequence $(...,b_3,b_2,b_1)$. By assumption, $\{\varphi(x,b'_n):n\in\mathbb{N}\}$ is realised in $A$.

Let $a\in A$ be such that the truth value of $\varphi(a,b'_n)$ alternates as many times as possible, subject to the constraint that $U\models\varphi(a,b'_n)$ eventually. We may assume there exists $m\in\mathbb{N}$ such that $U\models\varphi(a,b'_n)\leftrightarrow\neg\varphi(a,b'_{n+1})$, for all $n\in\{1,...,m-1\}$, and $U\models\varphi(a,b'_n)$, for all $n\geq m$. By compactness, there is some $k$ such that

\begin{enumerate}

\item there exists $a\in A$ such that $$U\models\bigwedge\limits_{n\in\{1,...,m-1\}}(\varphi(a,b'_n)\leftrightarrow\neg\varphi(a,b'_{n+1}))\wedge\varphi(a,b'_m)\wedge\varphi(a,b'_{m+2})\wedge...\wedge\varphi(a,b'_{m+k+1}),$$

\item for all $a$ as in (1), $U\models\varphi(a,b'_{m+1})$.
 
\end{enumerate}

Therefore, for the Morley sequence $(b_n)_{n\in\mathbb{N}}$, we have that

\begin{enumerate}

\item[(3)] there exists $a\in A$ such that $$U\models\varphi(a,b_1)\wedge...\wedge\varphi(a,b_k)\wedge\varphi(a,b_{k+2})\wedge\bigwedge\limits_{n\in\{k+2,...,k+m\}}(\varphi(a,b_n)\leftrightarrow\neg\varphi(a,b_{n+1})),$$

\item[(4)] for all $a$ as in (3), $U\models\varphi(a,b_{k+1})$.

\end{enumerate}

We shall show that there exist $d_{k+2},....,d_{k+m+1}\in M$ such that 

\begin{enumerate}

\item[(5)] there exists $a\in A$ such that $$U\models\varphi(a,b_1)\wedge...\wedge\varphi(a,b_k)\wedge\varphi(a,d_{k+2})\wedge\bigwedge\limits_{n\in\{k+2,...,k+m\}}(\varphi(a,d_n)\leftrightarrow\neg\varphi(a,d_{n+1})),$$

\item[(6)] for all $a$ as in (5), $U\models\varphi(a,b_{k+1})$.

\end{enumerate}

By Lemma \ref{A lemma}, there is a formula $\gamma(y_{k+2},...,y_{k+m+1})\in \text{tp}(b_{k+2},...,b_{k+m+1}/Mb_1...b_k)$ such that (5) holds for any $(d_{k+2},...,d_{k+m+1})\in \gamma(N)$. By (4) and compactness, there is then a formula $\gamma(y_{k+2},...,y_{k+m+1})\in \text{tp}(b_{k+2},...,b_{k+m+1}/Mb_1...b_k)$ such that (5) and (6) hold for any $(d_{k+2},...,d_{k+m+1})\in \gamma(N)$. As $\text{tp}(b_{k+2},...,b_{k+m+1}/Mb_1...b_k)$ is finitely realised in $M$, there exist $d_{k+2},....,d_{k+m+1}\in M$ such that (5) and (6) hold. 

Now let $z=(y_1,...,y_k)$ and let $\theta(x,z)$ be the formula $$\varphi(x,y_1)\wedge...\wedge\varphi(x,y_k)\wedge\varphi(x,d_{k+2})\wedge\bigwedge\limits_{n\in\{k+2,...,k+m\}}(\varphi(a,d_n)\leftrightarrow\neg\varphi(a,d_{n+1})).$$ By (5) and (6), there exists $a\in A$ such that $U\models\theta(a,b_1,...,b_k)$ and, for any such $a$, $U\models\varphi(a,b_{k+1})$. Let $\varphi'(x,y)$ be the formula $$\varphi(x,y)\wedge\varphi(x,d_{k+2})\wedge\bigwedge\limits_{n\in\{k+2,...,k+m\}}(\varphi(a,d_n)\leftrightarrow\neg\varphi(a,d_{n+1})).$$ Then $\varphi'(x,y)$ has NIP. By (5), (6) and the indiscernibility of $(b_n)_{n\in\mathbb{N}}$ over $M$, $\{\varphi'(x,b_n):n\in\mathbb{N}\}$ is realised in $A$. It follows, by Proposition \ref{refined equivalence}, that $\{\varphi'(x,b'):b'\in N\text{ and }b'\models \text{tp}(b/M)\}$ is realised in $A$. Then $$\{\theta(x,b_1',...,b_k'):b_1',...,b_k'\in N\text{ and }(b_1',...,b_k')\models\text{tp}(b_1,...,b_k)\}$$ is realised in $A$. Furthermore, $\theta(x,z)$ has NIP with respect to $T$ since $\varphi(x,y)$ does. 

By Proposition \ref{fact 3}, there is a formula $\tau(z)\in \text{tp}(b_1,...,b_k/M)$ such that, for any sequence $(c_n)_{n\in \mathbb{N}}$ in $\tau(N)$ that is indiscernible over $M$, $\{\theta(x,c_n):n\in\mathbb{N}\}$ is realised in $A$. Then, by Proposition \ref{fact 4}, the set system $$(A,\{\theta(U,b''_1,...,b''_k):(b''_1,...,b''_k)\in\tau(N)\})$$ has the $(p',q')$-property for some $$p'\geq q'> \text{VC}^*(A,\{\theta(U,b''_1,...,b''_k):(b''_1,...,b''_k)\in\tau(N)\}).$$ By Theorem \ref{infinite pq}, there is a positive integer $n'$, depending only on $(p',q')$, such that there exist $a_1,...,a_{n'}\in A$ such that, for all $c\in \tau(N)$, $U\models\theta(a_i,c)$ for some $i\in\{1,...,n'\}$. 

Using the fact that $(b_n)_{n\in\mathbb{N}}$ is a strict Morley sequence over $M$, for any finite set $B\subseteq N^{|y|}$ of realisations of $\text{tp}(b/M)$ there is an automorphism $\sigma$ of $N$ over $M$ that maps $B$ to a set of realisations of $\text{tp}(b_{k+1}/Mb_1...b_k)$. Then there exists $a\in A$ such that $U\models\theta(a,\sigma^{-1}(b_1),...,\sigma^{-1}(b_k))$ and, for all such $a$ and all $b'\in B$, $U\models \varphi(a,b')$.

By compactness, there is a formula $\psi\in \text{tp}(b/M)$ such that, for any $b''_1,...,b''_{n'}\in \psi(N)$, there exists $c\in \tau(N)$ such that, for all $a\in A$, if $U\models \theta(a,c)$ then, for all $i\in \{1,...,n'\}$, $U\models\varphi(a,b''_i)$. If for each $i\in\{1,...,n'\}$ there exists $b''_i\in\psi(N)$ such that $U\models\neg\varphi(a_i,b''_i)$ then we obtain a contradiction. Therefore, for at least one $i\in\{1,...,n'\}$, $U\models\varphi(a_i,b'')$ for all $b''\in\psi(N)$. This completes the proof of Theorem \ref{refined version}.

\section{Dense pairs}

Recall that a complete theory $T$ is geometric if its models are infinite, the quantifier $\exists^\infty$ is eliminated and the algebraic closure operator satisfies the exchange principle (so, for each $M\models T$, $(M,\text{acl})$ is a pregeometry). Throughout this section, we assume $T$ is a complete geometric theory. Let $L$ be the language of $T$ and $P$ a new unary relation symbol. Let $L_P=L\cup\{P\}$. We denote by $\text{acl}_L$ algebraic closure in the sense of $L$. 

\begin{definition}\label{dense pair}
A dense pair of models of $T$ is an $L_P$-structure $(M,P(M))$ such that $P(M)\prec M$ as $L$-structures and, for every infinite $X\subseteq M$ such that $X$ is definable with parameters in the $L$-structure $M$, \begin{enumerate}\item $X\cap P(M)\neq \emptyset$ and \item for every finite $B\subseteq M$, $X\nsubseteq \text{acl}_L(P(M)\cup B)$.\end{enumerate}
\end{definition}

Let $T_P$ be the common $L_P$-theory of all dense pairs of models of $T$. It was shown in \cite{AlexanderBerenstein2010} that $T_P$ is complete. It was shown in \cite{Berenstein2011} that if $T$ has NIP then so does $T_P$. However, it was shown in \cite{Hieronymi2017} that there are distal geometric theories for which the theory of dense pairs is not distal, for example the theory of real-closed ordered fields. So, after Boxall and Kestner had proved Theorem \ref{mt version} for distal theories, it was natural to investigate whether or not it was true for the theories of dense pairs of distal geometric structures. 

Rakotonarivo gave a positive answer in his thesis \cite{Rakotonarivo2021} in the case where $\text{acl}_L=\text{dcl}_L$, using a version of Theorem \ref{refined version} for distal geometric theories which was also given in \cite{Rakotonarivo2021}. While the definable $(p,q)$-theorem for such theories is a special case of Kaplan's theorem, it remains an interesting application of the ideas behind Theorem \ref{refined version}. We present the argument here, in a slightly improved form which removes the need for $\text{acl}_L=\text{dcl}_L$.

Let $(M,P(M))\prec^* (N,P(N))\prec^* (U,P(U))\models T_P$. We recall some basic facts about dense pairs. A key concept is $P$-independence.

\begin{definition}
Let $B\subseteq U$. Then $B$ is said to be $P$-independent if, for every finite $C\subseteq B$, $\dim_L(C/B\cap P(U))=\dim_L(C/P(U))$. 
\end{definition}

The $P$-type of a tuple is the information of which coordinates lie in $P(U)$ and which do not. The following was proved in \cite{AlexanderBerenstein2010} (their Lemma 2.8). 

\begin{lemma}\label{basic}
Let $a$ and $b$ be (possibly infinite) $P$-independent tuples in $U$. Then $\text{tp}_{L_P}(a)=\text{tp}_{L_P}(b)$ provided $\text{tp}_L(a)=\text{tp}_L(b)$ and the substructures generated by $a$ and $b$ have the same $P$-type. 
\end{lemma}

The following lemma is an easy consequence of the definitions and the basic properties of $\text{acl}_L$-independence in geometric theories. It is essentially well known.

\begin{lemma}\label{lemma}
    Let $a$ be a finite tuple $U$. 

    \begin{enumerate}
        \item $Ma$ is $P$-independent if and only if, after reordering $a$, we have $a=efg$, where $e$ is $\text{acl}_L$-independent over $M\cup P(U)$, $f$ is a tuple in $P(U)$, $f$ is $\text{acl}_L$-independent over $P(M)$ and $g\in \text{acl}_L(Mef)$.
        
    \item In this case, the $P$-type of $g$ depends only on $\text{tp}_L(efg/M)$ together with the above-mentioned properties of $e$ and $f$. 
        \item If $Ma$ is not $P$-independent then there is a finite tuple $c\in P(U)$ such that $Mac$ is $P$-independent. 
    \end{enumerate}
\end{lemma}

From now on, assume $T$ has NIP. So $T_P$ does too. We present an argument for deducing Theorem \ref{mt version} for $T_P$ from Theorem \ref{refined version} for $T$. Let $\varphi(x,y)$ be an $L_P$-formula, possibly with parameters in $M$, and let $b\in N$. Suppose $\varphi(x,b)$ does not divide over $M$ (in the sense of $L_P$). Since $T_P$ has NIP, there exists $a\in U$ such that $\text{tp}_{L_P}(a/N)$ is invariant over $M$ and $(U,P(U))\models \varphi(a,b)$. By Lemma \ref{lemma}(3), we may assume $Ma$ is $P$-independent. Let $e,f$ and $g$ be as in Lemma \ref{lemma}(1).

\begin{lemma}\label{over P}
 The tuple $e$ is $\text{acl}_L$-independent over $N\cup P(U)$. 
\end{lemma}

\proof Let $r=\dim_L(e/MP(U))$ and $s=\dim_L(e/NP(U))$. Suppose $s<r$. Let $d$ be a tuple in $N$ such that $\dim_L(e/MdP(U))=s$. We may assume $d$ is $\text{acl}_L$-independent over $MP(U)$. Let $(d_n)_{n\in\mathbb{N}}$ be a strict Morley sequence for $\text{tp}_L(d/M)$ in $N$. Let $t=\dim_L(d/MP(U))$. Then $\dim_L(d_{n+1}/Md_1...d_n)=t$, for all $n\in\mathbb{N}$. By (2) of Definition \ref{dense pair}, we may assume $\dim_L(d_{n+1}/Md_1...d_nP(U))=t$, for all $n\in\mathbb{N}$. Then, by Theorem \ref{basic} and Lemma \ref{lemma}(2), $\text{tp}_{L_P}(d_n/M)=\text{tp}_{L_P}(d/M)$, for all $n\in\mathbb{N}$. Since $\text{tp}_{L_P}(a/N)$ is invariant over $M$, it follows that $\dim_L(e/Md_n)=s$, for all $n\in\mathbb{N}$. By the exchange principle, $\dim_L(d_n/Me)<t$, for all $n\in\mathbb{N}$. Then $$(r+1)t=\dim_L(d_1...d_{r+1}/M)=\dim_L(d_1...d_{r+1}/Me)+\dim_L(e/M)\leq (r+1)(t-1)+r$$ which is a contradiction. The result follows. \endproof

\begin{lemma}\label{in P}
The tuple $f$ is $\text{acl}_L$-independent over $P(N)$. 
\end{lemma}

\proof The proof is similar to that of the previous lemma. Let $k=\dim_L(f/P(M))$ and $s=\dim_L(f/P(N))$. Suppose $s<k$. Let $d$ be a tuple in $P(N)$ such that $\dim_L(f/P(M)d)=s$. We may assume that $d$ is $\text{acl}_L$-independent over $P(M)$. Let $(d_n)_{n\in\mathbb{N}}$ be a strict Morley sequence for $\text{tp}_L(d/P(M))$ in $P(N)$. Let $t=\dim_L(d/P(M))$. We then have $\dim_L(d_{n+1}/P(M)d_1...d_n)=t$, for all $n\in\mathbb{N}$. By (1) of Definition \ref{dense pair}, we may assume $d_n\in P(N)$, for all $n\in\mathbb{N}$. Then, by Theorem \ref{basic} and Lemma \ref{lemma}(2), $\text{tp}_{L_P}(d_n/M)=\text{tp}_{L_P}(d/M)$, for all $n\in\mathbb{N}$. Since $\text{tp}_{L_P}(a/N)$ is invariant over $M$, it follows that $\dim_L(f/P(M)d_n)=s$, for all $n\in\mathbb{N}$. By the exchange principle, $\dim_L(d_n/P(M)f)<t$, for all $n\in\mathbb{N}$. Then $$(k+1)t=\dim_L(d_1...d_{k+1}/P(M))=\dim_L(d_1...d_{k+1}/P(M)f)+\dim_L(f/P(M))$$ $$\leq (k+1)(t-1)+k$$ which is a contradiction. The result follows. \endproof

We are already assuming that $Ma$ is $P$-independent. By Lemma \ref{lemma}(3), we may also assume that $Mb$ is $P$-independent. Using Lemma \ref{lemma}(1), let $e'$ be $\text{acl}_L$-independent over $MP(U)$, $f'$ a tuple in $P(N)$ which is $\text{acl}_L$-independent over $P(M)$ and $g'$ a tuple in $\text{acl}_L(Me'f')$ such that $b=e'f'g'$. By Lemmas \ref{over P} and \ref{in P}, $ee'$ is $\text{acl}_L$-independent over $MP(U)$ and $ff'$ is $\text{acl}_L$-independent over $P(M)$. So, by Lemma \ref{lemma}(1), $Mab$ is $P$-independent. 

Let $\theta(x)$ and $\theta'(y)$ be $L$-formulas over $M$ that witness $a\in\text{acl}_L(Mef)$ and $b\in\text{acl}_L(Me'f')$ respectively. Let $q(x,y)$ be the partial subtype of $\text{tp}_{L_P}(ab/M)$ that contains $\theta(x)$ and $\theta(y)$ and asserts that $ee'$ is $\text{acl}_L$-independent over $MP(U)$, $ff'$ is a tuple in $P(U)$ and $ff'$ is $\text{acl}_L$-independent over $P(M)$. By Theorem \ref{basic}, Lemma \ref{lemma} and compactness, there exist an $L$-formula $\varphi'(x,y)$ over $M$ and some $\alpha\in q$ such that $$\models \forall x \forall y (\alpha(x,y)\rightarrow (\varphi(x,y)\leftrightarrow\varphi'(x,y)).$$ Say $e$ has length $r$ and $f$ has length $k$. Let $A$ be the set of all $a'\in U^{\left|x\right|}$ such that 

\begin{enumerate}
\item the first $r+k$ coordinates of $a'$ are $\text{acl}_L$-independent over $N$ and
\item $a'$ realises $\theta$.
\end{enumerate} 

\begin{lemma}
    The set $\{\varphi'(x,b'):b'\in N,b'\models \text{tp}_L(b/M)\}$ is realised in $A$.
\end{lemma}

\proof With $e',f', g'$ as above, say $e'$ has length $r'$ and $f'$ has length $k'$. Let $(b_n)_{n\in\mathbb{N}}$ be a strict Morley sequence for $\text{tp}_L(b/M)$ in $N$. For each $n$, we have $b_n=e'_nf'_ng'_n$, where $e'_n$ has length $r'$ and $f'_n$ has length $k'$. For each $n$, $e'_nf'_n$ is $\text{acl}_L$-independent over $Mb_1...b_{n-1}$ and $b_n$ realises $\theta'(y)$. By Definition \ref{dense pair}, we may assume that, for each $n\in\mathbb{N}$, $\dim_L(e'_n/Me'_1...e'_{n-1}P(U))=r'$ and $f'_n\in P(N)$. It then follows, by Lemma \ref{lemma}(2) and Theorem \ref{basic}, that $\text{tp}_{L_P}(b_n/M)=\text{tp}_{L_P}(b/M)$, for all $n\in\mathbb{N}$. Since $\text{tp}_{L_P}(a/N)$ is invariant over $M$, it follows that $\{\varphi'(x,b_n):n\in\mathbb{N}\}$ is realised in $A$ by $a$. It follows by Proposition \ref{refined equivalence} that $\{\varphi'(x,b'):b'\in N,b'\models \text{tp}_L(b/M)\}$ is realised in $A$. \endproof 

By Theorem \ref{refined version}, there is an $L$-formula $\psi'(y)\in \text{tp}_L(b/M)$ such that $\{\varphi'(x,b'):b'\in \psi'(N)\}$ is realised in $A$. Let $a'$ be a realisation. By Definition \ref{dense pair}, we may assume the first $r$ coordinates of $a'$ are $\text{acl}_L$-independent over $NP(U)$ and the next $k$ coordinates are in $P(U)$. Let $\psi\in \text{tp}_{L_P}(b/M)$ be such that, for all $b'\in N^{|y|}$, if $b'$ realises $\psi$ then $b'$ realises $\psi'$ and $a'b$ realises $\alpha$. Then $\{\varphi(x,b'):b'\in\psi(N)\}$ is realised in $A$ and so is consistent.

\bibliographystyle{plain}
\bibliography{papier}

\end{document}